\documentclass[12pt]{amsart}
\usepackage[active]{srcltx}
\usepackage{atveryend}
\makeatletter
\let\origcitation\citation
\AtEndDocument{\def\mycites{}%
  \def\citation#1{\g@addto@macro\mycites{#1^^J}\origcitation{#1}}}
\AtVeryEndDocument{\newwrite\citeout\immediate\openout\citeout=\jobname.cit
  \immediate\write\citeout{\mycites}\immediate\closeout\citeout}
\makeatother
\usepackage{enumerate}
\usepackage{cancel}
\usepackage{mathdots}
\usepackage{epsfig}
\usepackage{hyperref}
\usepackage{amsmath}
\usepackage{amssymb}
\usepackage{graphics} 
\usepackage{yfonts}
\usepackage[usenames,dvipsnames]{color}

\newtheorem{question}{Question}

\newtheorem{corollary}{Corollary}

\newtheorem*{remark*}{Remark}
\newtheorem{proposition}{Proposition}
\newtheorem{theorem}{Theorem}
\newtheorem*{theorem*}{Theorem}
\newtheorem{lemma}{Lemma}
\newtheorem*{example*}{Example} 

\newtheorem*{notational conventions*}{Notational Conventions} 
\newcommand{\bZ}{{\mathbb Z}}
\newcommand{\bD}{{\mathbb D}}
\newcommand{\bR}{{\mathbb R}}
\newcommand{\bN}{{\mathbb N}}

\title[Number systems with simplicity hierarchies II]{Number systems with simplicity hierarchies: a generalization of Conway's theory of surreal numbers II}
\author{Philip Ehrlich}

\address{Department of Philosophy\\ Ohio University\\ Athens\\ OH 45701\\ U.S.A.}
\email{ehrlich@ohio.edu}

\author{Elliot Kaplan}

\address{Department of Mathematics \\ University of Illinois\\ Urbana-Champaign \\IL 61801\\ U.S.A.}
\email{eakapla2@illinois.edu}

\begin{document}
\maketitle

\begin{abstract} 
In \cite{EH5}, the algebraico-tree-theoretic simplicity hierarchical structure of J. H. Conway's ordered field $\bf No$ of surreal numbers was brought to the fore and employed to provide necessary and sufficient conditions for an ordered field to be isomorphic to an initial subfield of $\bf No$, i.e. a subfield of $\bf No$ that is an initial subtree of $\bf No$. In this sequel to \cite{EH5}, analogous results for ordered abelian groups and ordered domains are established which in turn are employed to characterize the convex subgroups and convex subdomains of initial subfields of ${\bf No}$ that are themselves initial. It is further shown that an initial subdomain of ${\bf No}$ is discrete if and only if it is an initial subdomain of ${\bf No}$'s canonical integer part ${\bf Oz}$ of omnific integers. Finally, extending results of \cite{EH5}, the theories of divisible ordered abelian groups and real-closed ordered fields are shown to be the sole theories of ordered abelian groups and ordered fields all of whose models are isomorphic to initial subgroups and initial subfields of ${\bf No}$.

\end{abstract} 

\bigskip



\tableofcontents


\section{Introduction}

 	J. H.
Conway \cite{CO} introduced a real-closed field of \emph{surreal numbers} embracing the reals and the ordinals as well as a great
many less familiar numbers including  $ - \omega $,  $\omega /2$,  $1/\omega $,  $\sqrt \omega $
and  $e^\omega $, to name only a few. This particular real-closed field, which
Conway calls  ${\bf No}$, is so remarkably inclusive that, subject to the proviso that
numbers---construed here as members of ordered fields---be individually definable in
terms of sets of von Neumann-Bernays-G\"{o}del set theory with global choice (NBG)
\cite{ME}, it may be said to contain ``All Numbers Great and Small.'' In this respect, ${\bf No}$ bears much the same relation to ordered fields that the ordered field $\mathbb{R}$ of real numbers bears to Archimedean ordered fields (\cite{EH3}, \cite{EH5}, \cite{EH8}). 

In addition to its inclusive structure as an ordered field,  ${\bf No}$ has a rich \emph{simplicity hierarchical (\emph{or} s-hierarchical) structure} \cite{EH5}, \cite{EH4}, that depends upon
its structure as a \emph{lexicographically ordered full binary tree} and arises
from the fact that the sums and products of any two members of the tree are the simplest possible
elements of the tree consistent with  ${\bf No}$'s structure as an ordered group and an ordered
field, respectively, it being understood that  $x$ is \emph{simpler than}  $y$ just in case  $x$ is a
predecessor of  $y$ in the tree.

Among the striking s--hierarchical features of  ${\bf No}$ that emerged from \cite{EH5} is that much as the surreal numbers emerge from the empty set of surreal
numbers by means of a transfinite recursion that provides an unfolding of the entire spectrum of
numbers great and small (modulo the aforementioned provisos), the recursive process of defining
${\bf No}$'s arithmetic in turn provides an unfolding of the entire spectrum of ordered abelian
groups (ordered fields) in such a way that an isomorphic copy of every such
system either emerges as an initial subtree of ${\bf No}$ or is contained in a theoretically
distinguished instance of such a system that does. In particular, it was shown that every divisible
ordered abelian group (real-closed ordered field) is isomorphic to an \emph{initial subgroup
(initial subfield)} of  ${\bf No}$.

The divisible ordered abelian groups and real-closed ordered fields, however, do not exhaust the ordered abelian groups and ordered fields that are isomorphic to initial subgroups and subfields of ${\bf No}$.  For example, every 2-divisible Archimedean ordered abelian group has an initial isomorphic copy in ${\bf No}$, as does and every Archimedean ordered field \cite[Theorem 8]{EH5}, and these groups and fields of course are not in general divisible or real-closed. In the case of ordered fields, more generally, in \cite[Theorem 18]{EH5} it was shown that:

\vspace{5pt}

\noindent
\emph{An ordered field is isomorphic to an initial subfield of {\bf No} if and only if it is isomorphic to a truncation closed, cross sectional subfield of a power series field  $\bR((t^\Gamma))_{\bf On}$ where $\Gamma$ is isomorphic to an initial subgroup of {\bf No}.}

\vspace{5pt}

The present paper is a sequel to \cite{EH5}. Following some preliminary material, in \S 5 and \S 6 we generalize for ordered abelian groups and ordered domains the just-said result for ordered fields, and in \S 7 we employ these generalizations to characterize the convex subgroups and convex subdomains of initial subfields of ${\bf No}$ that are themselves initial. We further show that an initial subdomain of ${\bf No}$ is discrete if and only if it is an initial subdomain of ${\bf No}$'s canonical integer part ${\bf Oz}$ of \emph{omnific integers}. And, in \S 8, we extend results of \cite{EH5} by showing that the theories of divisible ordered abelian groups and real-closed ordered fields are the \emph{sole} theories of ordered abelian groups and ordered fields \emph{all} of whose models are isomorphic to initial subgroups and initial subfields of ${\bf No}$. Finally, in \S 9 we state a pair of open questions regarding s-hierarchical ordered algebraic systems that supplement another open question raised in \S 8.

	Throughout the paper, the underlying set theory is assumed to be NBG and as such by
\emph{class} we mean set or proper class, the latter of which, in virtue of the axioms of global
choice and foundation, always has the ``cardinality'' of the universe of sets. For additional
information on formalizing the theory of surreal numbers in NBG, we refer the reader to
\cite{EH2}.

\section{Preliminaries I: Lexicographically ordered binary trees and Surreal Numbers}

A \emph{tree}  $\left\langle {A, < _s } \right\rangle $ is a partially ordered class such that for
each  $x \in A$, the class  $\left\{ {y \in A:y < _s x} \right\}$ of \emph{predecessors} of $x$,
written  `$pr_A \left( x \right)$', is a set well ordered by  $ < _s $. A maximal subclass of $A$ well
ordered by  $ < _s $ is called a \emph{branch} of the tree. Two elements  $x$ and $y$ of $A$ are
said to be \emph{incomparable} if  $x \ne y$,  $x\not  < _s y$ and  $y\not  < _s x$. An
\emph{initial subtree} of  $\left\langle {A, < _s } \right\rangle $ is a subclass  $A'$ of  $A$ with
the induced order such that for each $x \in A'$,  $pr_{A'} \left( x \right) = pr_A \left( x \right)$.
The \emph{tree-rank} of  $x \in A$, written  `$\rho _A (x)$', is the ordinal corresponding to the
well-ordered set $\left\langle {pr_A \left( x \right), <_s}\right\rangle $; the  $\alpha $th
\emph{level} of  $A$ is $\left\langle {x \in A:\rho _A (x) = \alpha } \right\rangle$ ; and a root of
$A$ is a member of the zeroth level. If  $x,y \in A$, then  $y$ is said to be an \emph{immediate
successor} of  $x$ if  $x < _s y$ and $\rho _A (y) = \rho _A (x) + 1$; and if  $(x_\alpha
)_{\alpha  < \beta }$ is a chain in $A$ (i.e., a subclass of  $A$ totally ordered by  $ < _s $), then
$y$ is said to be an \emph{immediate successor of the chain} if  $x_\alpha   < _s y$ for all
$\alpha  < \beta$ and  $\rho _A (y)$ is the least ordinal greater than the tree-ranks of the members
of the chain. The \emph{length} of a chain  $(x_\alpha  )_{\alpha  < \beta }$ in  $A$ is the
ordinal  $\beta $.

	A tree $\left\langle {A, < _s } \right\rangle $ is said to be \emph{binary} if each member
of  $A$ has at most two immediate successors and every chain in $A$ of limit length has at most
one immediate successor. If every member of  $A$ has two immediate successors and every
chain in  $A$ of limit length (including the empty chain) has an immediate successor, then the
binary tree is said to be \emph{full}. Since a full binary tree has a level for each ordinal, the
universe of a full binary tree is a proper class.

Following \cite[Definition 1]{EH5}, a binary tree  $\left\langle {A, < _s } \right\rangle $ together with a total ordering  $ < $ defined
on  $A$ will be said to be \emph{lexicographically ordered} if for all $x,y \in A$,  $x$ is
incomparable with  $y$ if and only if  $x$ and  $y$ have a common predecessor lying between
them (i.e. there is a  $z \in A$ such that $z<_s x$, $z<_s y$ and either $x<z<y$ or $y<z<x$). The appellation ``lexicographically ordered" is motivated by the fact that: $\left\langle {A, < , < _s } \right\rangle $ is a lexicographically ordered binary tree if and only if $\left\langle {A, < , < _s } \right\rangle $ is isomorphic to an initial ordered subtree of the \emph{lexicographically ordered canonical full binary tree} $\langle B, < _{lex\left( B \right)} , < _B \rangle $, where $B$ is the class
of all sequences of  $ - $s and  $ + $s indexed over some ordinal, $x < _B y$ signifies that
$x$ is a proper initial subsequence of  $y$, and $(x_\alpha  )_{\alpha  < \mu }  < _{lex\left(
B \right)} (y_\alpha  )_{\alpha  < \sigma }$ if and only if $x_\beta   = y_\beta $ for all $\beta  < $
some  $\delta $, but  $x_\delta   < y_\delta$, it being understood that  $ -  < $ \emph{undefined}
$ <  + $  \cite[Theorem 1]{EH5}. 

\begin{notational conventions*}
{\rm Let  $\left\langle {A, < , < _s } \right\rangle $ be a
lexicographically ordered binary tree. If  $\left( {L,R} \right)$ is a pair of subclasses of  $A$ for
which every member of  $L$ precedes every member of  $R$, then we will write `$L < R$'.
Also, if  $x$ and  $y$ are members of  $A$, then `$x < _s y$' will be read ``$x$ \emph{is simpler
than} $y$''; and if there is an  $x \in I = \left\{ {y \in A:L < \left\{ y \right\} < R} \right\}$ such
that  $x < _s y$ for all  $y \in I - \left\{ x \right\}$, then we will denote this \emph{simplest
member of}  $A$ \emph{lying between the members of}  $L$ \emph{and the members of} $R$
by `$\left\{ {L\, | \, R} \right\}$'. Finally, by `$L_{s\left( x \right)} $' we mean  $\left\{ {a \in A:a <
_s x\;{\rm{and}}\;a < x} \right\}$ and by `$R_{s\left( x \right)} $' we mean  $\left\{ {a \in A:a <
_s x\;{\rm{and}}\;x < a} \right\}$.}

\end{notational conventions*}

	The following three propositions collects together a number of properties of, or results about,
lexicographically ordered binary trees that will be appealed to in subsequent portions of the
paper.

\begin{proposition}\cite[Theorem 2]{EH5}\label{P:Bc}
Let  $\left\langle {A, < , < _s } \right\rangle $ be a lexicographically ordered binary tree. (i) For
all  $x \in A$,  $x = \left\{ {L_{s\left( x \right)} \, | \, R_{s\left( x \right)} } \right\}$; (ii) for all  $x,y
\in A$,  $x < _s y$ if and only if  $L_{s\left( x \right)}  < \left\{ y \right\} < R_{s\left( x \right)} $
and  $y \ne x$; (iii) for all  $x \in A$ and all  $L,R \subseteq A$,  $x = \left\{ {L \, | \, R} \right\}$ if
and only if  $L$ is cofinal with  $L_{s\left( x \right)} $ and  $R$ is coinitial with  $R_{s\left( x
\right)} $ if and only if $L < \{x\} < R$ and $\{ y \in A: L < \{y\} < R\} \subseteq \{ y \in A: L_{s(x)} < \{y\} < R_{s(x)}\}$.

\end{proposition}

Let $\left\langle {{\bf{No}}, < , < _s } \right\rangle $ be the \emph{lexicographically ordered binary tree of surreal numbers} constructed in any of the manners found in the literature (\cite{EH4}, \cite{EH5}, \cite{EH5.1}, \cite{EH8}, \cite{vdDE}, \cite{SCST}), including simply letting $\left\langle {{\bf{No}}, < , < _s } \right\rangle =\langle B, < _{lex\left( B \right)} , < _B \rangle $.\footnote{Conway's cut construction \cite{CO} and the related construction based on Cuesta Dutari cuts introduced in \cite{EH1} (and adopted in \cite{AE1} and \cite{AE2}), do not include {\bf{No}}'s lexicographically ordered binary tree structure. However, as was noted in \cite[page 257]{EH4}, they admit relational extensions to the ordered tree structure vis-\'a-vis the definition: for all  $x = (L,R) , y \in  {\bf{No}}$, $x<_sy$ if and only if $L < \{y\} < R$ and $y \neq x$. 

The identification $\left\langle {{\bf{No}}, < , < _s } \right\rangle =\langle B, < _{lex\left( B \right)} , < _B \rangle $, which is employed in \cite{EH5}, \cite{vdDE},\cite{KM} and \cite{BM}, is simply a relational extension of the familiar (non tree-theoretic) construction of {\bf{No}} based on \emph{sign-expansions}---the members of $B$---introduced by Conway \cite[page 65]{CO}, and made popular by Gonshor \cite{GO}.} Central to the development of the s-hierarchical theory of surreal numbers is the following result where a lexicographically ordered binary tree  $\left\langle {A, < , < _s } \right\rangle $ is said
to be \emph{complete} \cite[Definition 6]{EH5}, if whenever  $L$ and  $R$ are subsets
of  $A$ for which  $L < R$, there is an  $x \in A$ such that  $x = \left\{ {L\, | \, R} \right\}$. 

\begin{proposition}\cite[Theorem 4]{EH5}\cite{EH5.2}\label{P:Bd}
A lexicographically ordered binary tree is complete if and only if it is full if and only if it is isomorphic to  $\left\langle {{\bf{No}}, < , < _s } \right\rangle $.

\end{proposition}

An immediate consequence of Proposition \ref{P:Bd} is 

\begin{proposition}\label{P:eta}

Let $\left\langle {A, < , < _s } \right\rangle $ be a lexicographically ordered binary tree. $\left\langle {A, < _s } \right\rangle $ is full if and only if $\left\langle {A, < } \right\rangle $ is an $\eta_{\bf {On}}$-\emph{ordering} (i.e. whenever  $L$ and  $R$ are subsets
of  $A$ for which  $L < R$, there is an  $x \in A$ such that $L <{x}< R$).

\end{proposition}

\section{Preliminaries II: Conway names}
\label{sec:theorems}

Let $\mathbb{D}$ be the set of all surreal numbers having finite tree-rank, and 
$$\mathbb{R}=\mathbb{D}\cup  \left \{ \left \{ L\, | \, R \right \}:\left ( L,R \right )\textrm{ is a
Dedekind gap in }\mathbb{D} \right \}.$$

	The following result regarding the structure of
$\mathbb{R}$ is essentially due to Conway \cite[pages 12, 23-25]{CO}.

\begin{proposition}\label{P:Cb}

$\mathbb{R}$  (with  $ + , - , \cdot $ and  $ < $ defined \`{a} la  ${\bf{No}}$) is
isomorphic to the ordered field of real numbers defined in any of the more familiar ways,
$\mathbb{D}$ being  ${\bf{No}}$'s ring of dyadic rationals (i.e., rationals of the form  $m/2^n
$ where  $m$ and  $n$ are integers);  $n = \left\{ {0,\dots,n - 1 \, | \, \varnothing } \right\}$ for each
positive integer  $n$,  $ - n = \left\{ {\varnothing \, | - \left( {n - 1} \right),\dots,0} \right\}$ for each
positive integer  $n$,  $0 = \left\{ {\varnothing \, | \, \varnothing } \right\}$, and the remainder of the
dyadics are the arithmetic means of their left and right predecessors of greatest tree-rank; e.g.,
$1/2 = \left\{ {0 \, | \, 1} \right\}$. The systems of natural numbers and integers so defined are henceforth denoted $\mathbb{N}$ and $\mathbb{Z}$, respectively.
\end{proposition}

${\bf No}$'s canonical class $\bf On$ of ordinals consists of the members of the ``rightmost" branch of $\left\langle {{\bf{No}}, < , < _s } \right\rangle $, i.e. the  unique branch of $\left\langle {{\bf{No}}, < , < _s } \right\rangle $ whose members satisfy the condition: $x<y$ if and only if $x<_sy$. In those formulations where surreal numbers are pairs $(L,R)$ of sets of previously defined surreal numbers (\cite{CO}, \cite{EH8}, \cite{AE1}), the ordinals are the surreal numbers of the form $(L,\varnothing)$; and in the formulation \cite{GO} where surreal numbers are sign-expansions (see \S 2), the ordinals are the sequences (including the empty sequence) consisting solely of +s.

	A striking s--hierarchical feature of  ${\bf No}$ is that every surreal number can
be assigned a canonical ``proper name'' (or normal form) that is a reflection of its characteristic
s--hierarchical properties. These \emph{Conway names}, as we call them, are expressed as
formal sums of the form  $\sum\nolimits_{\alpha  < \beta } {\omega ^{y_\alpha  } .r_\alpha} $
where  $\beta $ is an ordinal,  $\left( {y_\alpha} \right)_{\alpha  < \beta } $ is a strictly decreasing
sequence of surreals, and  $\left( {r_\alpha  } \right)_{\alpha  < \beta } $ is a sequence of nonzero
real numbers, the Conway name of an ordinal being just its Cantor normal form \cite[pages 31-33]{CO},\cite[\S3.1 and \S5]{EH5}.

The surreal numbers having Conway names of the form $\omega^y$ are called \emph{leaders} since they denote the simplest positive members of the various Archimedean classes of ${\bf No}$.  More formally, they may be inductively defined by formula

\begin{equation}
\omega^{y} = \left\{ 0, n\omega^{y^L}\, |\, \frac{1}{2^n}\omega^{y^R} \right\},
\end{equation}

\noindent
where $n$ ranges over the positive integers, and $a^L$ and $a^R$ range over the elements of $L_{s(y)}$ and $R_{s(y)}$, respectively.

There are a number of significant
relations between surreal numbers that are reflected in terms of relations
between their respective Conway names. The following collection of such results, which are known from the literature, will be appealed to in the subsequent discussion. 

\begin{proposition}\cite[Theorems 11 and 15]{EH5}
\label{thrm15ehr01}
(i) For all $x,y \in {\bf No}$, $\omega^{x}<_s \omega^{y}$ if and only if $x<_{s} y$;
(ii) $\sum_{\alpha < \mu} \omega^{y_\alpha} \cdot r_\alpha <_s \sum_{\alpha < \beta} \omega^{y_\alpha} \cdot r_\alpha$ whenever $\mu <_s \beta$;\\
\begin{equation*}
\begin{split}
(iii) \sum_{\alpha < \beta} \omega^{y_\alpha} \cdot r_\alpha= \left\{\sum_{\alpha < \mu} \omega^{y_\alpha} \cdot r_\alpha+\omega^{y_\mu}\cdot\right. & \left.\left(r_\mu-\frac{1}{2^n}\right)\ \right  |\\
& \left.\sum_{\alpha < \mu} \omega^{y_\alpha} \cdot r_\alpha+\omega^{y_\mu}\cdot\left(r_\mu+\frac{1}{2^n}\right)\right\},
\end{split}
\end{equation*}
if $\beta$ is a limit ordinal (where $n$ and $\mu$ range over all positive integers and all ordinals less than $\beta$, respectively).
\end{proposition}

We shall also appeal to the following compilation of results regarding Conway names which, while new to the literature, essentially consists of corollaries of results from the first author's analysis of the surreal number tree \cite{EH7} or improvements (based on that analysis) of a result of Gonshor \cite[Lemma 5.8(a)]{GO}.

\begin{lemma}
\label{squeeze}
Suppose $a,b,r \in \bR$ and $x, y \in {\bf No}$. Then: (i) $\omega^y \cdot a <_s \omega^y \cdot b$ whenever $a <_s b$.  (ii) If either $r \in \mathbb{R}-\mathbb{D}$, or $r \in \mathbb{D}-\mathbb{Z}$ and $L_{s(y)}= \varnothing$, then $$ \omega^{y}\cdot r= \{ \omega^{y}\cdot r^L\, |\, \omega^{y}\cdot r^R\};$$ moreover, in virtue of (i), $\omega^y \cdot r^L<_s \omega^y \cdot r$ and $\omega^y \cdot r^R <_s \omega^y \cdot r$ for all $r^L \in L_{s(r)}$ and $r^R \in R_{s(r)}$. (iii) If $r \in \mathbb{D}-\mathbb{Z}$ and $L_{s(y)} \neq \varnothing$, then $$ \omega^{y}\cdot r= \{ \omega^{y}\cdot r^L+\omega^{y^L}\cdot n\, |\, \omega^{y}\cdot r^R-\omega^{y^L}\cdot n\},$$ where $n$ ranges over $\mathbb{N}$; moreover, $ \omega^{y}\cdot r^L+\omega^{y^L}\cdot n\ <_s \omega^{y}\cdot r$ and $ \omega^{y}\cdot r^R-\omega^{y^L}\cdot n\ <_s \omega^{y}\cdot r$, for all  $y^L \in L_{s(y)}$, $r^L \in L_{s(r)}$ and $r^R \in R_{s(r)}$. 
(iv) For all $n \in \mathbb{N}$,   $\frac{1}{2^n}\omega^y <_s \omega^{x}$ if $y \in R_{s(x)}$.
\end{lemma}

\begin{proof}
(i) follows immediately from \cite[Theorems 3.13 and 3.16]{EH7}, by considering a number $c \in \bR - \bD$ such that $b=c$ or $b <_s c$. Necessarily, $\omega^y \cdot c$ is the immediate successor of a chain of limit length having a cofinal subchain of the form $(\omega^y\cdot c_n )_{n<\omega}$ where $c$ is the immediate successor of $(c_n)_{n<\omega}$. As $a \in (c_n)_{n<\omega}$ and either $b \in (c_n)_{n<\omega}$ or $b=c$, it is the case that $\omega^y\cdot a \in (\omega^y\cdot c_n )_{n<\omega}$ and either $\omega^y \cdot b \in (\omega^y\cdot c_n )_{n<\omega}$ or $\omega^y \cdot b = \omega^y\cdot c$. In either case, $\omega^y\cdot a <_s \omega^y \cdot b$. For (ii), if $r \in \mathbb{R}-\mathbb{D}$, the result can be proved from (i) by simply forming the specified cut in the just-said cofinal subchain; and if $L_{s(y)}= \varnothing$ and $r \in \mathbb{D}-\mathbb{Z}$, the result follows from (i) and the second part of Gonshor's Lemma 5.8(a) from \cite{GO}. (iii) follows from the first part of Gonshor's Lemma 5.8(a) from \cite{GO} and repeated applications of \cite[Theorem 3.18(ii)]{EH7}. For (iv), suppose without loss of generality that $x$ is the immediate left successor of $y$; for if it is not, the immediate left successor of $y$ is in $R_{s(x)}$ and if we show that $\frac{1}{2^n}\omega^y$ is simpler than the immediate left successor of $y$ for each $n$, then it follows that $\frac{1}{2^n}\omega^y$ is simpler than $x$ as well. Consider the chain $pr_{\textbf{No}}(\omega^x)$ of predecessors of $\omega^x$. Clearly, this is a chain of limit length, as $\omega^x$ is the immediate successor of another surreal number if and only if $x=0$, in which case $R_{s(x)} = \varnothing$. By \cite[Theorems 3.13 and 3.16]{EH7}, we may conclude that $pr_{\textbf{No}}(\omega^x)$ contains a cofinal chain of the form $(\omega^{y}\cdot a_n)_{n<\omega}$ where $a_n=\frac{1}{2^n}$ for each $n$. As this chain is contained in $pr_{\textbf{No}}(\omega^x)$, $\frac{1}{2^n}\omega^y <_s \omega^{x}$ for each $n$.
\end{proof}

Let $\bR((t^\Gamma))_{\bf On}$ be the ordered group (ordered domain; ordered field) of power series (defined \'a la Hahn \cite{H}) consisting of all formal power series of the form $\sum_{\alpha < \beta} r_\alpha t^{y_\alpha}$ where $(y_\alpha)_{\alpha < \beta \in {\bf On}}$ is a possibly empty descending sequence of elements of an ordered class (ordered monoid; ordered abelian group) $\Gamma$ and $r_\alpha \in \bR - \{0\}$ for each $\alpha < \beta$. $\bR((t^\Gamma))_{\bf On}$ is a set (often simply written $\bR((t^\Gamma))$) if $\Gamma$ is a set, and a proper class otherwise. An element $x \in \bR((t^\Gamma))_{\bf On}$ is said to be a \textit{proper truncation} of $\sum_{\alpha < \beta} r_\alpha t^{y_\alpha} \in \bR((t^\Gamma))_{\bf On}$ if $x = \sum_{\alpha < \sigma} r_\alpha t^{y_\alpha}$ for some $\sigma < \beta$. A subgroup (subdomain; subfield) $A$ of $\bR((t^\Gamma))_{\bf On}$ is said to be \textit{truncation closed} if every proper truncation of every member of $A$ is itself a member of $A$. A subgroup (subdomain; subfield) $A$ of $\bR((t^\Gamma))_{\bf On}$ is said to be \textit{cross sectional} if $\{t^y : y \in \Gamma \} \subseteq A$. For a truncation closed, cross sectional subgroup (subdomain; subfield) $A$ of $\bR((t^\Gamma))_{\bf On}$, the set  $\bR_y = \{r\ \in \bR : rt^y \in A \}$ is an Archimedean ordered group (domain; field) which we will call the \textit{$y$-coefficient group (domain; field)} of $A$.

\begin{proposition}[\cite{EH5}, \cite{EH8}]
\label{prop 5}
There is an isomorphism of ordered groups from {\bf No} onto $\bR((t^{\bf No}))_{\bf On}$ that sends each surreal number $\sum_{\alpha < \beta} \omega^{y_\alpha} \cdot r_\alpha$ to $\sum_{\alpha < \beta} r_\alpha t^{y_\alpha}$. The isomorphism is in fact an isomorphism of ordered domains and, hence, of ordered fields.

\end{proposition}

\section{Preliminaries III: s-hierarchical ordered structures }

Following \cite[Definition 2]{EH5}, $\left\langle A,+,<,<_{s},0\right\rangle $ is said
to be an {\it s-hierarchical ordered group} if (i) $\left\langle
A,+,<,0\right\rangle $ is an ordered abelian group; (ii) $\left\langle
A,<,<_{s}\right\rangle $ is a lexicographically ordered binary tree; and
(iii) for all $x,y\in A$ 
$$
x+y=\left\{ x^{L}+y,x+y^{L} \, | \, x^{R}+y,x+y^{R}\right\} \text{.\medskip } 
$$

 $\left\langle A,+,\cdot ,<,<_{s},0,1\right\rangle $ will be said to be an {\it s-hierarchical ordered domain }if (i) $\left\langle
A,+,\cdot ,<,0,1\right\rangle $ is an ordered domain; (ii) $\left\langle
A,+,<,<_{s},0\right\rangle $ is an s-hierarchical ordered group; and (iii)
for all $x,y\in A$ 
\begin{eqnarray*}
\!\!\!xy &=&\{x^{L}y+xy^{L}-x^{L}y^{L},x^{R}y+xy^{R}-x^{R}y^{R} \,| \\
&&\quad \quad \qquad \qquad \qquad \text{\quad }
x^{L}y+xy^{R}-x^{L}y^{R},x^{R}y+xy^{L}-x^{R}y^{L}\}\text{.}
\end{eqnarray*}

\noindent
Moreover, $\left\langle A,+,\cdot ,<,<_{s},0\right\rangle $ will be
said to be an {\it s-hierarchical ordered $K$-module} if
(i) $K$ is an s-hierarchical ordered domain, (ii) $A$ is an s-hierarchical
ordered group, and (iii) $A$ is an ordered $K$-module in which for all 
$x\in K$ and all $y\in A$%
\begin{eqnarray*}
\!\!\!xy &=&\{x^{L}y+xy^{L}-x^{L}y^{L},x^{R}y+xy^{R}-x^{R}y^{R} \, | \\
&&\quad \quad \qquad \qquad \qquad \text{\quad }%
x^{L}y+xy^{R}-x^{L}y^{R},x^{R}y+xy^{L}-x^{R}y^{L}\}\text{.}
\end{eqnarray*}

s-hierarchical ordered domains and modules are generalizations of the s-hierarchical ordered fields and vector spaces introduced in \cite{EH5}. In virtue of Conway's field operations, $\left\langle {\bf {No}},+,\cdot ,<,<_{s},0,1\right\rangle $ is an s-hierarchical ordered domain. In fact, it is (up to isomorphism) the unique universal and unique maximal s-hierarchical ordered domain (in the sense of \cite[page 1239]{EH5}). Moreover, if $K$ is an s-hierarchical ordered subdomain of {\bf {No}}, then {\bf {No}} is an s-hierarchical ordered $K$-module. Furthermore, extending the argument for s-hierarchical ordered subfields and subspaces of {\bf {No}} from \cite[page 1236]{EH5}, it is evident that

\begin{proposition}

The s-hierarchical ordered subdomains and submodules of 
{\bf {No}} coincide with the initial subdomains and the initial submodules of {\bf {No}}, respectively. 

\end{proposition}

Extending the notation employed in \cite{EH5}, if $A$ is an ordered module and $B \subseteq A$, then by $(B)_{A}$ we mean \emph{the ordered submodule of $A$ generated by $B$.} 

The next preparatory result is a modest generalization of \cite[Theorem 6]{EH5}.

\begin{lemma}
\label{thrm6anlg}
Let $M'$ be an s-hierarchical ordered $K$-module and $M$ be an initial submodule of $M'$. If $(L,R)$ is a partition of $M$ and $b=\{L\,| \, R\}^{M'}$, then $(M\cup \{b\})_{M'}$ is an initial submodule of $M'$.
\end{lemma}

\begin{proof}
Except for replacing the references to ``$K$-vector spaces" and ``subspaces" with references to ``$K$-modules" and ``submodules", the proof is the same as the proof of  \cite[Theorem 6]{EH5}.
\end{proof}

\section{Initial subgroups and submodules of {\bf No}}
\label{sec:groups}

To fully characterize the initial subgroups of $\bf {No}$, we must first characterize the initial subgroups of $\bR$.

\begin{lemma}
\label{initarchgroups}
An ordered group $G$ is an initial subgroup of $\bR$ if and only if either $G=\{0\}$, $\mathbb{D} \subseteq G$ or $G=\{\frac{z}{2^m}: z \in \bZ\}$ for some $m \in \bN$.
\end{lemma}

\begin{proof}
It follows from the definition of $\bR$ and Proposition \ref{P:Cb} that the subgroups of $\bR$ specified in the statement of the lemma are initial. Now suppose $G$ is a nontrivial initial subgroup of $\bR$ that does not contain $\bD$. Then there is a greatest $m \in \bN$ such that $\frac{z}{2^m} \in G$ for some $z \in \bZ$. Using closure under subtraction and the fact that we must always have 1 and, therefore $\bZ$ in any nontrivial initial subgroup of $\textbf{No}$, it follows that $\frac{1}{2^m} \in G$ and, by closure under addition and subtraction, $\{\frac{z}{2^m}: z \in \bZ\} \subseteq G$. But since $m$ is the greatest member of $\bN$ for which $\frac{z}{2^m} \in G$ for some $z \in \bZ$, the inclusion is not proper, thereby proving the lemma.
\end{proof}

Our final  preparatory result follows immediately from the definitions.

\begin{proposition}
\label{modgengroup}
If $G$ is a cross sectional, truncation closed subgroup of $\bR((t^\Gamma))_{\bf On}$ and 
\begin{equation*}
Z = \left\{ \sum_{\mu < \nu}r_\mu t^{y_\mu} \in G : \nu \textit{ is an infinite limit ordinal and }r_0 = 1 \right\},
\end{equation*}
then $\{rt^y : y \in \Gamma,\ r \in \bR_y\}\cup Z$ constitutes a class of generators for $G$ considered as a $\bZ$-module.
\end{proposition}

As is noted above, {\bf No}  is isomorphic to $\bR((t^{\bf{No}}))_{\bf On}$. We now prove more generally

\begin{theorem}
\label{initgroups}
An ordered abelian group is isomorphic to an initial subgroup of {\bf No} if and only if it is isomorphic to a truncation closed, cross sectional subgroup $G$ of a power series group $\bR((t^\Gamma))_{\bf On}$, where (i) $\Gamma$ is isomorphic to an initial ordered subclass of {\bf No},
(ii) every $y$-coefficient group $\bR_y$ of $G$ is an initial subgroup of $\bR$, and
(iii) $\bD \subseteq \bR_y$ for all $x, y \in \Gamma$ where $y \in R_{s(x)}$.
\end{theorem}

\begin{proof}

 Let $A$ be an initial subgroup of {\bf No}, $\textit{Lead}(A)$ be the class of leaders in $A$, and $\Gamma=\{y: \omega^y \in \textit{Lead}(A)\}$. Following   \cite[Definition 14]{EH5},  a class $B$ of surreal numbers is said to be \emph{approximation complete} if $\sum_{\alpha < \sigma} \omega^{y_\alpha} \cdot r_\alpha \in B$ whenever $\sum_{\alpha < \beta} \omega^{y_\alpha} \cdot r_\alpha \in B$ and $\sigma < \beta$. Since $A$ is initial, it follows from parts (i) and (ii) of Proposition \ref{thrm15ehr01} that $A$ is approximation complete and $\Gamma$ is initial. Moreover, since $A$ is a group, for each $y \in \Gamma$ the set $ \bR_y=\{r\in \bR: \omega^y\cdot r \in A\}$ is a subgroup of $\bR$. Furthermore, since $A$ is initial, it follows from Lemma \ref{squeeze}(i) that $\bR_y$ is itself initial. Now suppose $x, y \in \Gamma$ where $y \in R_{s(x)}$. In virtue of Lemma \ref{squeeze}(iv), $\frac{1}{2^n}\omega^y <_s \omega^{x}$ for each  $n \in \bN$. Therefore, $\frac{1}{2^n} \in \bR_y$ for each $n \in \bN$; and thus, since groups are closed under addition and subtraction, $\bD \subseteq \bR_y$. Accordingly, by appealing to the restriction to $A$ of the isomorphism specified in Proposition \ref{prop 5}, it is evident that 
$G=\{\sum_{\alpha < \beta} r_\alpha t^{y_\alpha}  \in \bR((t^\Gamma))_{\bf On}: \sum_{\alpha < \beta} \omega^{y_\alpha} \cdot r_\alpha \in A\}$ is a truncation complete, cross sectional subgroup of  $\bR((t^\Gamma))_{\bf On}$ having the requisite properties listed in the statement of the theorem, thereby establishing the ``only if" portion of the theorem.

Aspects of the ``if" portion of the proof borrow from the first author's proof of \cite[Theorem 18]{EH5}. However, whereas the latter proof concerns an ordered subfield $F$ of $\bR((t^\Gamma))_{\bf On}$, which is treated as an ordered vector space over the Archimedean ordered field $\{r \in\bR: rt^0 \in F\}$, here we are concerned with an ordered subgroup $G$ of $\bR((t^\Gamma))_{\bf On}$ that we may only assume to be an ordered $\bZ$-module, which complicates the argument. To keep the argument largely self-contained, however, we repeat with modifications portions of the earlier proof.

Let $G$ be a subgroup of $\bR((t^\Gamma))_{\bf On}$ satisfying the conditions specified in the statement of the theorem and let $A$ be the isomorphic copy of $G$ in {\bf No} that is the image of the restriction to $G$ of the inverse of the mapping referred to in Proposition \ref{prop 5}. That is, let $A=\{\sum_{\alpha < \beta} \omega^{y_\alpha} \cdot r_\alpha \in {\bf No}:     \sum_{\alpha < \beta} r_\alpha t^{y_\alpha} \in G \}$. To show that $A$ is an initial subgroup of \textbf{No}, it suffices to show that $\langle A, <_s | A \rangle$ is an initial subtree of $\langle \textbf{No}, <_s \rangle$. We do this by induction on $\Gamma$.

Let $a_0, \ldots, a_\alpha, \ldots ( \alpha < \beta)$ be a well-ordering of $\Gamma$ such that $\rho_{\textbf{No}}(a_\mu) \leq \rho_{\textbf{No}}(a_\nu)$ whenever $\mu < \nu < \beta$. We consider $A$ as an ordered $\bZ$-module. Let $A_\alpha$ be the submodule of $A$ containing $0$ as well as all of the elements in $A$ with exponents only from $\Gamma_\alpha = \{ a_\delta : \delta \leq \alpha \}$. We see that, considered as an ordered $\bZ$-module, $A=\bigcup_{\alpha < \beta} A_\alpha$. Notice also that since $a_0=0$, $A_0=\bR_0$, which, by condition (ii), is an initial subgroup of \textbf{No}. Therefore, $\langle A_0, <_s|A_0 \rangle$ is an initial subtree of \textbf{No}. To complete  the proof that $A$ is an initial subtree of \textbf{No}, it remains to show that for all $0<\alpha < \beta$, $\langle A_\alpha, <_s | A_\alpha \rangle$ is an initial subtree of \textbf{No}, if $\langle \bigcup_{\mu < \alpha} A_\mu, <_s | \bigcup_{\mu < \alpha} A_\mu\rangle$ is an initial subtree of \textbf{No}.

Accordingly, let $0<\alpha < \beta$ and suppose $\langle \bigcup_{\mu < \alpha} A_\mu, <_s | \bigcup_{\mu < \alpha} A_\mu\rangle$ is an initial subtree of \textbf{No}. Also, let $Z_\alpha =:$
\begin{equation*}
 \left\{ \sum_{\mu < \nu} \omega^{y_\mu} \cdot r_\mu \in A_\alpha - \bigcup_{\mu < \alpha} A_\mu : \textit{$\nu$ {\rm is an infinite limit ordinal and} $r_0 = 1$}\right\},
\end{equation*}
and let $b_0,\ldots,b_\sigma,\ldots (\sigma < \tau)$ be a well-ordering of  $\{ \omega^{a_\alpha} \cdot r_\alpha : r_\alpha \in \bR_{a_\alpha}\} \cup Z_\alpha$ such that for all $\gamma,\ \delta < \tau$:
\begin{enumerate}
\item if $b_\gamma$, $b_\delta \in \{ \omega^{a_\alpha} \cdot r_\alpha : r_\alpha \in \bR_{a_\alpha}\}$, then $b_\gamma < b_\delta$ only if $\rho_{\textbf{No}}(b_\gamma) \leq \rho_{\textbf{No}}(b_\delta)$;
\item if $b_\gamma \in \{ \omega^{a_\alpha} \cdot r_\alpha : r_\alpha \in \bR_{a_\alpha}\}$ and $b_\delta \in Z_\alpha$, then $b_\gamma < b_\delta$;
\item if $b_\gamma$, $b_\delta \in Z_\alpha$ and the initial sequence of ordinals over which the exponents in $b_\gamma$ are indexed is contained in the initial sequence of ordinals over which the exponents in $b_\delta$ are indexed, then $b_\gamma < b_\delta$.
 
\end{enumerate}

By appealing to Proposition \ref{modgengroup} (and recalling that $(X)_{A}$ denotes the ordered submodule of $A$ generated by $X$), it is easy to see that $\{ \omega^{a_\alpha} \cdot r_\alpha : r_\alpha \in \bR_{a_\alpha}\} \cup Z_{\alpha} \cup\bigcup_{\mu < \alpha} A_\mu$ is a class of generators for $A_\alpha$, and hence that, $A_\alpha = \bigcup_{\sigma < \tau} B_\sigma$, where $B_0 = \left( \{b_0\} \cup \bigcup_{\mu < \alpha} A_\mu \right)_A$ and $B_\sigma = \left( \{b_\sigma\} \cup \bigcup_{\delta < \sigma} B_\delta \right)_A$ for $0 < \sigma < \tau$. Thus to show that $A_\alpha$ is an initial subtree of \textbf{No}, it suffices to show that $B_\sigma$ is an initial subtree of \textbf{No} for each $\sigma < \tau$. Moreover, since $B_\sigma = \bigcup_{\delta < \sigma} B_\delta$, whenever $b_\sigma \in \bigcup_{\delta < \sigma} B_\delta$, henceforth we need only consider those $b_\sigma \notin \bigcup_{\delta < \sigma} B_\delta$.

First, note that $b_0 = \omega^{a_\alpha}$. Moreover, since $\Gamma$ is assumed to be initial, both $L_{s(a_\alpha)}$ and $R_{s(a_\alpha)} \subseteq \{a_\delta : \delta < \alpha \}$. Let $a^L$ and $a^R$ be typical elements of $L_{s(a_\alpha)}$ and $R_{s(a_\alpha)}$, respectively. It follows from Equation (1) (see \S3) that
\begin{equation*}
b_0 = \omega^{a_\alpha} = \left\{ 0, n\omega^{a^L} \, | \, \frac{1}{2^n}\omega^{a^R} \right\}
\end{equation*}
where $n$ ranges over the positive integers. As every element of $\bigcup_{\mu < \alpha} A_\mu - \{ 0 \}$ is Archimedean equivalent to a unique member of $\{\omega^{a_\delta} : \delta < \alpha \} \subseteq \bigcup_{\mu < \alpha} A_\mu - \{ 0 \}$, there is be a unique partition $(L_\alpha, R_\alpha)$ of $\bigcup_{\mu < \alpha} A_\mu$ where $L_\alpha < R_\alpha$, $\{ 0, n\omega^{a^L} \}$ is cofinal with $L_\alpha$ and $\{ \frac{1}{2^n}\omega^{a^R} \}$ is coinitial with $R_\alpha$. In virtue of the well ordering, each $n\omega^{a^L}$ is in $\bigcup_{\mu < \alpha} A_\mu$ as is each $\omega^{a^R}$. Plainly then, $\{n\omega^{a^L} \} \subseteq L_\alpha$. Moreover, in virtue of the well ordering and the fact that condition (iii) in the statement of the theorem requires $\bD \subseteq \bR_{a^R}$ for each $a^R$, it follows that $\{ \frac{1}{2^n}\omega^{a^R} \} \subseteq R_\alpha$. But then by Proposition \ref{P:Bc}, $b_0 = \{L_\alpha\, |\, R_\alpha \}$; and so, by Lemma \ref{thrm6anlg}, $B_0$ is an initial subtree of \textbf{No}.

If $\{ \omega^{a_\alpha} \cdot r_\alpha : r_\alpha \in \bR_{a_\alpha}-\bZ\}=\varnothing$, we turn to $Z_\alpha$. Otherwise, we take the first remaining $b_\sigma\in \{ \omega^{a_\alpha} \cdot r_\alpha : r_\alpha \in \bR_{a_\alpha}-\bZ\}$ for which $b_\sigma \notin \bigcup_{\delta < \sigma} B_\delta$ and observe that, in virtue of the nature of our well ordering,  $b_\sigma=\omega^{a_\alpha}\cdot r_\alpha$ for some $r_\alpha \in \bD - \bZ$ since every $r_\alpha \in \bR - \bD$ has predecessors in $\bD - \bZ$. Moreover, either $L_{s(a_\alpha)} = \varnothing$ or $L_{s(a_\alpha)} \neq \varnothing$. If $L_{s(a_\alpha)} = \varnothing$, then by the relevant portion of Lemma \ref{squeeze}(ii), we have 
\begin{equation*}
b_\sigma = \{ \omega^{a_\alpha}\cdot r_\alpha^L \, | \, \omega^{a_\alpha}\cdot r_\alpha^R\},
\end{equation*}
where $\omega^{a_\alpha}\cdot r_\alpha^L$, $\omega^{a_\alpha}\cdot r_\alpha^R <_s b_\sigma$; and this together with the nature of our well ordering implies $\omega^{a_\alpha}\cdot r_\alpha^L$, $\omega^{a_\alpha}\cdot r_\alpha^R \in \bigcup_{\delta < \sigma} B_\delta$. Furthermore, if $L_{s(a_\alpha)} \neq \varnothing$, then by Lemma \ref{squeeze}(iii), we have 
\begin{equation*}
b_\sigma = \{ \omega^{a_\alpha}\cdot r_\alpha^L+\omega^{a_\alpha^L}\cdot n \, | \,\omega^{a_\alpha}\cdot r_\alpha^R-\omega^{a_\alpha^L}\cdot n\},
\end{equation*}
where $ \omega^{y}\cdot r^L+\omega^{y^L}\cdot n\ <_s \omega^{y}\cdot r$ and $ \omega^{y}\cdot r^R-\omega^{y^L}\cdot n\ <_s \omega^{y}\cdot r$, for all  $y^L \in L_{s(y)}$, $r^L \in L_{s(r)}$, $r^R \in R_{s(r)}$ and $n \in\mathbb{N}$; and, again, this together with the nature of our well ordering implies that all those options are contained in $\bigcup_{\delta < \sigma} B_\delta$. Thus, in virtue of Proposition \ref{P:Bc} (iii), in both cases there is a partition $(L_\sigma', R_\sigma')$ of $\bigcup_{\delta < \sigma} B_\delta$ such that $b_\sigma = \{ L_\sigma'\, |\, R_\sigma' \}$; and so, by Lemma \ref{thrm6anlg}, $B_\sigma$ is an initial subtree of \textbf{No}. This portion of the proof is repeated until each $\omega^{a_\alpha}\cdot r_\alpha$ for some $r_\alpha \in \bD - \bZ$ has been dealt with.

Next, if $\{ \omega^{a_\alpha} \cdot r_\alpha : r_\alpha \in \bR_{a_\alpha}-\bD\}=\varnothing$, we turn to $Z_\alpha$. If not, we take the first remaining $b_\sigma\in \{ \omega^{a_\alpha} \cdot r_\alpha : r_\alpha \in \bR_{a_\alpha}-\bD\}$ for which $b_\sigma \notin \bigcup_{\delta < \sigma} B_\delta$ and argue exactly as we did in the case where $r_\alpha \in \bD - \bZ$ and $L_{s(a_\alpha)} = \varnothing$ except we appeal to the relevant portion of Lemma \ref{squeeze}(ii) concerned with members of $\bR - \bD$. This portion of the proof is repeated until each $\omega^{a_\alpha}\cdot r_\alpha$ for some $r_\alpha \in \bR - \bD$ has been dealt with.

Finally, if $Z_{\alpha }=\varnothing $, we are
finished; if not, let $0<\sigma <\tau $ and, as our induction hypothesis,
suppose $\bigcup_{\delta <\sigma }B_{\delta }$ is an initial subtree of $\bf No$. 
Since $0<\sigma <\tau $, $b_{\sigma }$ has a Conway name of the form $
\sum\nolimits_{\alpha <\,\pi }\omega ^{y_{\alpha }}.r_{\alpha }$, where $\pi 
$ is an infinite limit ordinal and $r_{0}=1$. Moreover, by part (ii) of Proposition \ref{thrm15ehr01}, $
b_{\sigma }=\left\{ L \, | \, R\right\}$ where 
\[
L=\left\{ \sum\nolimits_{\alpha <\,\mu }\omega ^{y_{\alpha }}.r_{\alpha
}+\omega ^{y_{\mu }}.\left( r_{\mu }-\frac{1}{2^{n}}\right) \right\}
_{0<n<\omega ,\mu <\pi } 
\]
and 
\[
R=\left\{ \sum\nolimits_{\alpha <\,\mu }\omega ^{y_{\alpha }}.r_{\alpha
}+\omega ^{y_{\mu }}.\left( r_{\mu }+\frac{1}{2^{n}}\right) \right\}
_{0<n<\omega ,\mu <\pi .} 
\]
But since, by construction, $L\cup R\subseteq \bigcup_{\delta <\sigma
}B_{\delta }$ and $b_{\sigma }\notin \bigcup_{\delta <\sigma }B_{\delta }$,
there is a partition $\left( L_{\sigma }^{\prime },R_{\sigma }^{\prime
}\right) $ of $\bigcup_{\delta <\sigma }$ $B_{\delta }$ such that 
\[
b_{\sigma }=\left\{ L_{\sigma }^{\prime } \, | \, R_{\sigma }^{\prime }\right\}
\text{.} 
\]
Therefore, by virtue of the induction hypothesis and Lemma \ref{thrm6anlg}, $B_{\sigma }$
is an initial subtree of ${\bf {No}}$. Thus, by induction, $B_{\sigma }$ is an
initial subtree of $\bf No$ for each $\sigma <\tau $, and so $A_{\alpha }$ and,
hence, $\left\langle A,<_{s}|A\right\rangle $ are initial subtrees of $\bf No$;
thereby proving the theorem.
\end{proof}

\begin{remark*}

Using the axiom of choice or global choice (if the class is a proper class), it is a routine matter to prove that every ordered class is isomorphic to an initial ordered subclass of {\bf {No}}.
This might seem to suggest that in the statement of Theorem \ref{initgroups} one may omit the assumption that $\Gamma$ is isomorphic to an initial ordered subclass of {\bf No}. However, the presence of condition (iii) precludes the omission of that statement.

\end{remark*}

\subsection{Densely and discretely ordered initial subgroups of {\bf No}}

A nontrivial ordered group $\langle G, <, +, 0\rangle$ is said to be \textit{discrete or discretely ordered} if it contains a least positive member, and it is said to be \textit{dense or densely ordered} if for all $a,b\in G$ where $a<b$ there is a $c\in G$ such that $a<c<b$. A nontrivial ordered group $\langle G, <, +, 0\rangle$ is dense if and only if it is not discrete.

The following proposition provides a simple means of distinguishing between nontrivial initial subgroups of {\bf No} that are discrete and those that are dense.

\begin{proposition}
\label{grouptypes}

An initial subgroup $G$ of {\bf No} is discrete if and only if there is a member of $G$ of the form $\frac{1}{2^n}\omega^{-\alpha}$ (where $n \in \bN$ and $\alpha \in On$) having no left immediate successor in $G$.
\end{proposition}

\begin{proof}
Note that elements of \textbf{No} of the form $\frac{1}{2^n}\omega^{-\alpha}$ are precisely those having 0 as their sole left predecessor. We must show that for any $g \in G$, $g$ is the least positive element of $G$ if and only if $L_{s(g)}=\{0\}$ and $g$ has no left immediate successor.

First, suppose there is a least positive $g \in G$ for which $L_{s(g)}\neq \{0\}$. If $0\notin L_{s(g)}$, then $g$ is not positive, a contradiction. If there is an $a\in L_{s(g)}$ where $a\neq 0$, then $a$ is a positive element less than $g$, another contradiction.

Next, notice that any least positive $g \in G$ must have no immediate left successor, since, if $g$ has a left immediate successor, this successor must be a positive element less than $g$.

To show the other direction, suppose $G$ is initial. Further suppose $g$ has no left immediate successor, $L_{s(g)}=\{0\}$ and there is an $a\in G$ where $0<a<g$. Since $G$ is lexicographically ordered, it follows that for any $x,y \in G$ where $x<y$, $x$ is incomparable with $y$ if and only if $x$ and $y$ have a common predecessor $z$ such that $x<z<y$ (see \S2). Clearly, $a$ must be incomparable with $g$, as the only surreal less than $g$ and comparable with $g$ is zero. Therefore, they must have a common predecessor $z$ such that $a<z<g$, but by the assumption that $L_{s(g)}=\{0\}$, $z$ must equal 0 and $a$ must be negative, which is impossible.
\end{proof}

\section{Initial subdomains of \textbf{No}}
\label{sec:domains}
We now turn to the characterization of initial subdomains of $\bf {No}$ beginning with the initial subdomains of $\bR$.

\begin{lemma}
\label{archdom}
The initial subdomains of $\bR$ are $\bZ$ and the subdomains of $\bR$ containing $\bD$. Every initial subdomain of $\bf {No}$ is an extension of an initial subdomain of $\bR$.

\end{lemma}

\begin{proof}
Let $K$ be a subdomain of $\bR$.  If $K=\bZ$, then by Proposition \ref{P:Cb} $K$ is initial.  Now suppose $K \neq \bZ$. If $K$ is initial, then $\bZ\subset K$ and, by Proposition \ref{P:Cb}, $K$ must contain some element of the form $z+\frac{1}{2}$ where $z \in \bZ$. By subtracting $z$, we see that $\frac{1}{2} \in K$. But since the domain $\bD$ is generated by $\frac{1}{2}$, $\bD\subseteq K$. Finally, if $\bD\subseteq K$, then $K$ is initial since every every predecessor of a member of $\bR-\bD$ is a member of $\bD$, and $\bD$ is initial. The second part of the lemma is trivial.
\end{proof}

\begin{theorem}
\label{initdoms}
An ordered domain is isomorphic to an initial subdomain of {\bf No} if and only if
it is isomorphic to a truncation closed, cross sectional subdomain $K$ of a power series domain $\bR((t^\Gamma))_{\bf On}$, where $\Gamma$ is isomorphic to an initial submonoid of {\bf No}, every $y$-coefficient domain $\bR_y$ of $K$ is an initial ordered subdomain of $\bR$, and $\bD \subseteq \bR_y$
for any $x, y \in \Gamma$ where $y \in R_{s(x)}$.
\end{theorem}

\begin{proof}
First note that the above conditions are precisely the conditions for initial groups, with the exception of the stipulation that $\Gamma$ must be form a monoid. The latter condition is necessary since, if the image of the domain $K$ is cross sectional, then for all $x, y \in \Gamma$, $t^x$ and $t^y$ are in $K$, so $t^x\cdot t^y = t^{x+y}$ is in $K$ and $x+y$ is in $\Gamma$. Of course, the rest of the conditions are necessary for integral domains as well, as they are necessary for abelian groups. In order to show that the conditions are also sufficient, we may treat $K$ as a $\bZ$-module and repeat the second part of the proof of Theorem \ref{initgroups}. 
\end{proof}

\subsection{Densely ordered initial subdomains of \textbf{No}}

An ordered domain is said to be \textit{dense} if its ordered additive group is dense.

\begin{corollary}
\label{initdensedoms}
A densely ordered domain is isomorphic to an initial subdomain of {\bf No} if and only if it is isomorphic to a truncation closed, cross sectional subdomain $K$ of a power series domain $\bR((t^\Gamma))_{\bf On}$ where $\Gamma$ is isomorphic to an initial submonoid of {\bf No} and $\bD$ is an initial subdomain of $K$.
\end{corollary}

\begin{proof}
Let $K$ be a dense initial subdomain of {\bf No}. In light of Theorem \ref{initdoms}, to prove the corollary it suffices to show that $\bD$ is an initial subdomain of $K$ if and only if (i) every $y$-coefficient domain $\bR_y$ of $K$ is an initial ordered subdomain of $\bR$, and (ii) $\bD \subseteq \bR_y$
for any $x, y \in \Gamma$ where $y \in R_{s(x)}$. Suppose $\bD$ is an initial subdomain of $K$. Then $\bD \subseteq \bR_0$. Moreover, since $K$ is cross sectional, $\bR_0 \subseteq \bR_y$ for all $y \in \Gamma$, since $t^y \in K$ and $r\cdot t^0 \in K$ for all $r \in \bR_0$, which implies $r\cdot t^y \in K$ for all $r \in \bR_0$. This implies (ii) is satisfied, which along with Lemma \ref{archdom} implies (i) is satisfied as well. Now suppose (i) and (ii) are the case. If $\bD$ is not an initial subdomain of $K$, then by (i) and Lemma \ref{archdom}, $\bR_0=\bZ$. Moreover, since $K$ is both dense and initial, the simplest member of {\bf No} lying between $0$ and $1$, namely $\frac{1}{2}$, is in $K$. But this implies $\bD$ is an initial subdomain of $K$, contrary to assumption. 
\end{proof}

The following result, which is a special case of Corollary \ref{initdensedoms}, is the aforementioned result (see \S 1) categorizing the initial subfields of ${\bf No}$ established in \cite{EH5}. Since the special case is about ordered fields, the ordered monoid $\Gamma$ must be an ordered abelian group and the reference to $\bD$ may be deleted since every ordered field is an extension of an isomorphic copy of $\bD$, the latter of which is initial in {\bf No}.

\begin{corollary}
\label{initfields}
An ordered field is isomorphic to an initial subfield of {\bf No} if and only if it is isomorphic to a truncation closed, cross sectional subfield of a power series field $\bR((t^\Gamma))_{\bf On}$ where $\Gamma$ is isomorphic to an initial subgroup of {\bf No}.

\end{corollary}

\subsection{ Discretely ordered initial subdomains of \textbf{No}}
An ordered domain is said to be \textit{discrete} if its additive group is discrete. Accordingly, an ordered domain is discrete if and only if it is not dense. The least positive member of an ordered domain is its multiplicative identity 1.  

A discrete subdomain $A$ of an ordered field $B$ is said to be an \emph{integer part} if every member of $B$ is at most a distance 1 from a member of $A$. Conway introduced a canonical integer part \textbf{Oz} of \textbf{No} 
consisting of the surreal numbers of the form $x = \{x-1\, |\, x+1\}.$ These \emph{omnific integers}, as Conway calls them, are precisely the surreal numbers having Conway names of the form $\sum_{\alpha < \beta} \omega^{y_{\alpha}} \cdot r_{\alpha}$ where $y_{\alpha} \geq 0$  for all $\alpha < \beta$, and $r_{\alpha}$ is an integer if $y_{\alpha} = 0$.

As we will now see, \textbf{Oz} is in fact (up to isomorphism) the unique discrete s-hierarchical ordered domain containing an initial isomorphic copy of every discrete s-hierarchical ordered domain. 

\begin{theorem}
\label{omnificdiscrete}
An initial subdomain of {\bf No} is discrete if and only if it is an initial subdomain of {\bf Oz}.
\end{theorem}

\begin{proof}
Suppose $K$ is an initial subdomain of \textbf{Oz}. As \textbf{Oz} is an initial subdomain of \textbf{No} \cite [p. 3: Note 2]{EH7}, $K$ must be as well. To see that $K$ is discrete, suppose on the contrary there is an element in $K$ between 0 and 1. Since $\frac{1}{2}= \{0 \, | \, 1\}$,  it follows that $\frac{1}{2} \in K$. But this is impossible since $\frac{1}{2} \notin \textbf{Oz}$, as $\{\frac{1}{2}-1\, |\, \frac{1}{2}+1\}=0$. For the converse, suppose that there is a discrete initial subdomain $K$ of \textbf{No} containing some element $a\notin \textbf{Oz}$. Let $\sum_{\alpha < \beta} \omega^{y_{\alpha}} \cdot r_{\alpha}$ be the Conway name of $a$. Also let $b = \sum_{\alpha < \beta} \omega^{y_{\alpha}} \cdot r'_{\alpha}$ where $r'_\alpha=r_\alpha$ if $y_\alpha > 0$, $r'_\alpha= 0$ if $y_\alpha<0$ and $r'_\alpha$ is largest integer less than $r_\alpha$ if $y_\alpha=0$. Note that $b \in \textbf{Oz}$, so $b = \{b-1\, |\, b+1\}$. Moreover, since $b-1 < a < b+1$, $b <_s a$; and so, as $K$ is initial and $a\in K$, $b\in K$. As $K$ is a domain, $a-b\in K$; but $0<a-b<1$, so $\frac{1}{2}<_s a-b$ and, therefore, $\frac{1}{2} \in K$, which contradicts the assumption that $K$ is discrete.
\end{proof}
\section{Initial subgroups and subdomains that are convex}
Among the important subgroups and subdomains of ordered groups and fields are those that are convex. Using Theorems \ref{initgroups} and  \;\ref{initdoms}, we now identify the convex subgroups of initial subgroups of \textbf{No} that are themselves initial as well as the convex subdomains (i.e. the valuation rings) of initial subfields of \textbf{No} that are likewise initial. As we shall see, unlike the convex subgroups and subdomains of ordinary ordered groups and ordered fields,  the initial convex subgroups and subdomains of initial subgroups and subfields of \textbf{No} are always well ordered by inclusion. 

Let $A$ be a nontrivial initial subgroup of \textbf{No} and $\textbf{On}(A)=:\textbf{On}\cap{A}$ be its subtree of ordinals. Following \cite[Definition 19]{EH5}, $A$ is said to be $\alpha\textit{-Archimedean}$ if $\alpha$ is the height of $\textbf{On}(A)$. A nontrivial initial subgroup of \textbf{No} is $\alpha\text{-Archimedean}$ if and only if for each $x\in A$ there is a $\beta \in \textbf{On}(A)$ such that $-\beta < x < \beta$ \cite[Theorem 24]{EH5}. 
	
Ordinals of the form $\omega^\phi$ are said to be \textit {additively indecomposable} since they are precisely the ordinals $\lambda$ such that $\mu + \nu < \lambda$ for all ordinals $\mu,\nu < \lambda$, and ordinals of the form $\omega^{\omega^\phi}$ are said to be \textit{multiplicatively indecomposable} since they are precisely the ordinals $\lambda >1$ such that $\mu\nu < \lambda$ for all ordinals $\mu,\nu < \lambda$, where the just-said sums and products of ordinals are the familiar Cantorian operations. Every nontrivial initial subgroup (resp. initial subdomain) is $\omega^\phi$-Archimedean for some nonzero ordinal (resp. nonzero additively indecomposable ordinal) $\phi$; moreover, $A$ is Archimedean if and only if $A$ is $\omega$-Archimedean \cite[Theorem 24]{EH5}.

Let $A$ be an $\omega^{\phi}$-Archimedean initial subgroup of \textbf{No} and for each nonzero ordinal $ \tau \le \phi$, let $$A[\omega^{\tau}] =: \{x \in A: -\alpha < x < \alpha \;\text{for some}\; \alpha < \omega^{\tau}\}.$$

\begin{proposition}
\label{pp}
Let $A$ be an $\omega^{\phi}$-Archimedean initial subgroup of {\bf No} and $\tau$ be a nonzero ordinal $\le \phi \le \emph{\bf{On}}$.\newline
\rm{(i)} $A[\omega^{\tau}]=A$ {\em if and only if} $\tau = \phi$.\newline
\rm{(ii)} \em $A[\omega^{\tau}] = \{x \in A: -n\alpha < x < n\alpha \;\text{{\em for some}}\; n \in \mathbb{N} \; \text{{\em and some}} \;\alpha < \omega^{\tau}\}$.\newline
\rm{(iii)} \em The class of ordinals $< \omega^{\tau}$ is a cofinal subclass of $A[\omega^{\tau}]$.\newline
\rm{(iv)} \em For each leader $\omega^y \in  A$ there is a unique leader $\omega^{\alpha} \in A$ where $\alpha$ is an ordinal $< \phi$ such that  $ \omega^y \le \omega^{\alpha}$ and  $\omega^{\alpha} \le_s \omega^y$. If $ \omega^y \neq \omega^{\alpha}$, then $\omega^{\alpha}$ is the least ordinal $> \omega^y$. Moverover, if $\omega^{x} <_s \omega^y$, then $\omega^x \le  \omega^{\alpha}$.
\end{proposition}

\begin{proof}
(i) is trivial, (ii) follows from the additively indecomposable nature of  $\omega^{\tau}$, and (iii) follows from the definition of $A[\omega^{\tau}]$ and the initial nature of $A$. For (iv), note that for each $z>0$ in \textbf{No}, there is a unique ordinal $\alpha$ such that $z \le \alpha$ and $\alpha \le_s z$; moreover, if $y \neq \alpha$, then $\alpha$ is the least ordinal $>y$.  But by \cite[Theorem 11]{EH5} and \cite[Proposition 3.6 (i)]{EH7} for all $z,y \in \textbf{No},\; z<_s y$ if and only if $\omega^z <_s  \omega^y$, which implies the first two parts of (iv). Finally, suppose $\omega^z <_s  \omega^y$. If $z$ is an ordinal, then plainly $\omega^z \le  \omega^\alpha$; and if $z$ is not an ordinal, then its sign-expansion begins with $\omega^\alpha$ pluses followed by a minus, which implies it is $< \omega^\alpha$, thereby completing the proof.
\end{proof}

\begin{theorem}
\label{GROUP}
Let $A$ be an $\omega^{\phi}$-Archimedean initial subgroup of \textbf{\emph{No}}. Then $K$ is a nontrivial initial convex subgroup of $A$ if and only if $K= A[\omega^{\tau}]$ for some additively indecomposable infinite ordinal $\omega^{\tau} \le \omega^{\phi}$. 
\end{theorem}

\begin{proof}
First suppose  $K$ is a nontrivial initial convex subgroup of $A$. Since $K$ is a nontrivial initial subgroup of $A$, $K$ is $\omega^{\tau}$-Archimedean for some nonzero ordinal $\tau \le \phi$. Moreover, since $K$ is a convex subgroup of $A$, $A[\omega^{\tau}] \subseteq{K}$. Furthermore, if $A[\omega^{\tau}] \neq{K}$, there is an $x\in K$ such that $A[\omega^{\tau}] < \{x\} <\omega^{\tau}$. But then $\omega^{\tau} <_s x$, which implies $K$ is not initial, since $\omega^{\tau} \notin K$; and so $K=A[\omega^{\tau}]$. 

Now suppose $K= A[\omega^{\tau}]$ for some additively indecomposable infinite ordinal $\omega^{\tau} \le \omega^{\phi}$. If $\tau=\phi$, then $K=A$, which is a trivial initial convex subgroup of $A$. Now suppose $1\le\tau < \phi$. In virtue of the definition of $A[\omega^{\tau}]$, $K$ is a convex subclass of $A$. Moreover, by Proposition \ref{pp}(ii) and the convex nature of $K$, every element of $A$ that is Archimedean equivalent to some $x \in K$ is likewise in $K$. Plainly then, $x+y \in K$ whenever $x, y \in K$, since $x+y$ is in the Archimedean class containing $\omega^z$, where $z$ is the maximal member in the supports of $x$ and the supports of $y$, which shows $K$ is a group. To establish $K$ is initial, first note that since $A$ is cross sectional and closed under truncation, it follows from the fact that $K$ is a convex subgroup of $A$ (in which every element of $A$ that is Archimedean equivalent to some $x \in K$ is likewise in $K$) that $K$ is also cross sectional and closed under truncation. In addition, since $K$ is a convex subgroup of $A$ and $A$ is initial, the $y$-coefficient groups of $K$ satisfy conditions (ii) and (iii) of Theorem \ref{initgroups}. Therefore, in virtue of Theorem \ref{initgroups}, to complete the proof it remains to show $\{y \in \textbf{No}: \omega^y \in K\}$ is an initial subtree of \textbf{No}. For this purpose, suppose $\omega^{y} \in K$ and further suppose $\omega^{x} <_s \omega^y$.  Since $x<_s y$ if and only if  $\omega^{x} <_s \omega^y$ for all $x,y \in \bf{No}$, it suffices to show $\omega^x\in K$. 
To this end, note that by Proposition \ref{pp}(iii), there is an ordinal $\beta \in K$ such that $\omega^y < \beta< \omega^\tau$. Moreover, by Proposition \ref{pp}(iv), there is an ordinal  $\omega^{\alpha}$ such that either $\omega^{\alpha}=\omega^y$ or $\omega^{\alpha}$ is the least ordinal $> \omega^y$ and $\omega^x \le  \omega^{\alpha}$. Plainly then, $\omega^x \le \omega^{\alpha} \leq \beta$. And so, since $K$ is a convex subgroup of $A$, $\omega^x  \in K$.
\end{proof}

\begin{theorem}
\label{DOMAIN}
Let $A$ be an $\omega^{\phi}$-Archimedean initial subfield of \textbf{\emph{No}}. Then $K$ is an initial convex subdomain of $A$ if and only if $K= A[\omega^{\tau}]$ for some multiplicatively indecomposable ordinal $\omega^{\tau} \le \omega^{\phi}$. 
\end{theorem}

\begin{proof}
Since $A$ is a field, it follows that for each $\omega^{y} \in K$, $\{r \in \mathbb{R}: \omega^{y}.r \in K\} = \{r \in \mathbb{R}: \omega^{0}.r \in K\}$ and so, for each leader $\omega^{y} \in K$, $\mathbb{D}$ is a subdomain of $\{r \in \mathbb{R}: \omega^{y}.r \in K\}$. Accordingly, in virtue of Theorem \ref{GROUP} and Theorem \ref{initdoms}, to complete the proof it suffices to show that $K$ (which is convex) is a subdomain of $A$ if and only if $K= A[\omega^{\tau}]$ for some multiplicatively indecomposable ordinal $\omega^{\tau} \le \omega^{\phi}$. If $\omega^{\tau}$ is not multiplicatively indecomposable, there are ordinals $\alpha,\beta \in K$ where $\alpha\beta>\omega^{\tau}$ and, hence, $\alpha\beta \notin K$, which shows $K$ is not a domain. Now suppose $\omega^{\tau}$ is multiplicatively indecomposable. If $\tau=1$, $K$ consists of the finitely bounded members of $A$, in which case $K$ is obviously a subdomain of $A$. Moreover, if $1<\tau < \phi$, then $\{\omega^{\beta}: \beta < \tau\}$ is a cofinal subclass of $K$ without a greatest member. Accordingly, if $x,y \in K$, $|x |\le$ some $\omega^{\beta} \in K$ and  $|y| \le$ some $\omega^{\gamma} \in K$, where $\beta, \gamma$ are ordinals $<\tau$, and so $|xy| \le \omega^{\beta + \gamma}<\omega^{\tau}$. But then $\omega^{\beta + \gamma}$ and, hence, $|xy|$ are in $K$, which suffices to show $K$ is a domain.
\end{proof}

\begin{remark*}

{\rm Since every real-closed ordered field is isomorphic to an initial subfield of \textbf{No}, it is natural to inquire if this is true for real-closed ordered domains more generally, the latter of which coincide with the convex subdomains of real-closed ordered fields \cite{CD}. However, since the convex subdomains of real-closed ordered fields are not in general well-ordered by inclusion, Theorem \ref{DOMAIN} implies this is not the case.}

\end{remark*}

\section{Optimality results and an open question}

As we mentioned above, in \cite{EH5} it was shown that every divisible
ordered abelian group (real-closed ordered field) is isomorphic to an initial subgroup
(initial subfield) of  ${\bf No}$. The following result shows that in a important sense these results are optimal. 

Let $T^{OAG}$ and $T^{DOAG}$ be the theories of ordered abelian groups and divisible ordered abelian groups in the language $\{\leq,+,0\}$ of ordered additive groups, and let $T^{OF}$ and $T^{RCF}$ be the theories of ordered fields and real-closed ordered fields in the language $\{\leq,+,\cdot,0,1\}$ of ordered fields. 

\begin{theorem}
\label{OP}

\vspace{5pt}
\noindent
(i) If $T^{OAG}  \subseteq   T \subseteq T^{DOAG}$, then every model of $T$ is isomorphic to an initial subgroup of \textbf{\em {No}} if and only if $T = T^{DOAG}$.

\vspace{5pt}
\noindent
(ii) If $  T^{OF}  \subseteq   T \subseteq T^{RCF}$, then every model of $T$ is isomorphic to an initial subfield of \textbf{\em {No}} if and only if $T = T^{RCF}$.

\end{theorem}

\begin{proof}

In light of the above-mentioned results on divisible ordered abelian groups and real-closed fields from \cite{EH5}, it remains to consider the cases where $T \neq T^{DOAG}$ and $T \neq T^{RCF}$. Let $T^{OAG}  \subseteq   T \subset T^{DOAG}$ and $M$ be a countable model of $T$. There is an elementary chain $M_\alpha$ ($\alpha < {\bf On}$) of models of $T$ such that for each $\alpha $, $M_\alpha$ is an $\omega_{\alpha+1}$-saturated elementary extension of $M$ of power $2^{\omega_\alpha}$ \cite[Lemma 5.1.4]{CK}. Since each $M_\alpha$ is an $\eta_{\alpha+1}$-ordering  (\cite[Page 369]{CK}; also see \cite{ALK}), the union $M'$ of the chain is a model of $T$ that is an $\eta_{\bf On}$-ordering (see Proposition \ref{P:eta}). However, since {\bf {No}} is a lexicographically ordered full binary tree and no lexicographically ordered binary tree contains a proper initial ordered subtree that is isomorphic to itself \cite[Lemmas 1 and 2]{EH5}, it follows from Propositions \ref{P:Bd} and \ref{P:eta} that an initial ordered subtree of {\bf {No}} is an $\eta_{\bf On}$-ordering if and only if it is {\bf {No}} itself. But since {\bf {No}} is divisible and $M'$ is not, there is no isomorphism of $M'$ onto an initial subgroup of {\bf {No}}. Except for trivial modifications, the same argument applies to (ii), which suffices to prove the theorem. 
\end{proof}

The above proof of Theorem \ref{OP} makes critical use of class models. Using a classical result from the theory of saturated models \cite[Lemma 5.1.5]{CK} together with a natural generalization of Theorem 4 of (\cite{EH5}, \cite{EH5.2}), a variation of the above proof shows that an analog of Theorem \ref{OP} that applies solely to structures whose universes are sets can be established in NBG supplemented with the existence of an inaccessible cardinal $< {\bf On}$ or an instance of GCH. This naturally suggests:

\begin{question}

Can an analog of Theorem \ref{OP} that applies solely to ordered abelian groups and ordered fields whose universes are sets be established in \rm{NBG}?

\end{question}

An ordered field $K$ is said to be $n$-\emph{real-closed} \cite[page 327]{B} if every polynomial of degree at most $n$ admitting a root in a real closure of $K$ admits a root in $K$. Boughattas \cite{B} has shown that, for each positive integer $n$, there is a model of the theory of $n$-real-closed fields whose universe is a set that does not have an integer part. Therefore, since every initial subfield of {\bf {No}} has a canonical integer part \cite[Theorem 25]{EH5}, it follows that for each positive integer $n$, there is a model of the theory of $n$-real-closed fields whose universe is a set that is not isomorphic to an initial subfield of {\bf {No}}. This, however, does not provide a positive answer to the field portion of Question 1 since there are theories of ordered fields that are not equivalent to the theory of $n$-real-closed fields for any $n$. One can also find theories of ordered abelian groups having models whose universe are sets that are not isomorphic to initial subgroups of {\bf {No}}. For example, if we say that an ordered abelian group is $n$-\emph{divisible} if every element is divisible by $n$, then for each prime $p>2$ the $p$-divisible ordered abelian group generated by 1 is not isomorphic to an initial subgroup of {\bf{No}}. However, we are not aware of a proof that applies to every theory of ordered abelian groups lacking full divisibility, which leaves the group portion of Question 1 open as well.

\section{Further open questions}

Unlike discrete initial subdomains of {\bf{No}}, discrete initial subgroups of {\bf{No}} need not be subgroups of {\bf{Oz}}. A case in point is the subgroup of  {\bf{No}} consisting of all elements of the form $d+(a/\omega)$, where $d \in \bD$ and $a \in \bZ$. On the other hand, this discrete initial subgroup of {\bf{No}} is isomorphic to an initial subgroup of {\bf{Oz}}, as is evident from the mapping $f(d+(a/\omega))=\omega.d + a$ for all $d \in \bD$ and all $a \in \bZ$. This motivates

\begin{question}

Is every discrete initial subgroup of \emph{\bf{No}} isomorphic to an initial subgroup of \emph{\bf{Oz}}? 

\end{question}

\begin{question}
What is a set of conditions that are individually necessary and collectively sufficient for an arbitrary ordered monoid to be isomorphic to an initial submonoid of \emph{\bf{No}}?
\end{question}


\end{document}